\documentclass[12pt]{amsart}
\usepackage{latexsym} 
\usepackage{amsfonts} 
\usepackage{amssymb} 
\usepackage{color}
\usepackage{marginnote}
\usepackage{amscd,enumerate}
\usepackage{epsfig}
\usepackage{amsmath}
\usepackage[all]{xy}
\newcommand{\Z}{{\mathbb{Z}}}
\newcommand{\N}{{\mathbb{N}}}

\newcommand{\R}{{\mathbb{R}}}
\def\gr{\mathop{\partial_{\R}}}
\newtheorem{theo}{Theorem}
\newtheorem{pr}{Proposition}
\newtheorem{co}{Corollary}
\newtheorem{lm}{Lemma}
\newtheorem{re}{Remark}
\newtheorem{de}{Definition}

\def\rad{\mathop{\mathfrak{B}_{\text{gr}}}}
\def\a{\alpha}
\font\hu=msbm10

\font\got=eufb10
\font\gots=eufb8
\def\Ps{\text{\gots S}}
\def\P{\text{\got S}}

\def\proof#1{Proof. {#1} \hfill {$\square$}}

\def\lin#1{\langle #1 \rangle}
\def\lan{\mathop{\text{\rm lan}}}
\def\ran{\mathop{\text{\rm ran}}}
\def\Path{\mathop{\text{\rm Path}}}

\title[Centers of path algebras]{Centers of path algebras, Cohn and Leavitt path algebras}

\author[M. G. Corrales]{Mar\'{\i}a G. Corrales Garc\'{\i}a}
\address{M. G. Corrales Garc\'{\i}a:  Centro Regional Universitario de Cocl\'e: \lq\lq Dr. 
Bernardo Lombardo\rq\rq, Universidad de  Panam\'a.  Apartado Postal 0229. Penonom\'e, 
Provincia de Cocl\'e. Panam\'a.}
\email{mcorrales@ancon.up.ac.pa}
\author[D. Mart\'{\i}n]{Dolores Mart\'{\i}n Barquero}
\address{D. Mart\'{\i}n Barquero: Departamento de Matem\'atica Aplicada, Escuela T\'ecnica Superior de Ingenieros Industriales, Universidad de M\'alaga. 29071 M\'alaga. Spain.}
\email{dmartin@uma.es}

\author[C. Mart\'{\i}n]{C\'andido Mart\'{\i}n Gonz\'alez}
\address{C. Mart\'{\i}n Gonz\'alez:  Departamento de \'Algebra Geometr\'{\i}a y Topolog\'{\i}a, Fa\-cultad de Ciencias, Universidad de M\'alaga, Campus de Teatinos s/n. 29071 M\'alaga. Spain.}
\email{candido@apncs.cie.uma.es}

\author[M. Siles ]{Mercedes Siles Molina}
\address{M. Siles Molina: Departamento de \'Algebra Geometr\'{\i}a y Topolog\'{\i}a, Fa\-cultad de Ciencias, Universidad de M\'alaga, Campus de Teatinos s/n. 29071 M\'alaga.   Spain.}
\email{msilesm@uma.es}

\author[J. F. Solanilla]{Jos\'e F. Solanilla Hern\'andez}
\address{J. F. Solanilla Hern\'andez:  Centro Regional Universitario de Cocl\'e: \lq\lq Dr. 
Bernardo Lombardo\rq\rq,  Universidad de  Panam\'a. Apartado Postal 0229. Penonom\'e, 
Provincia de Cocl\'e. Panam\'a.}
\email{jsolanilla@ancon.up.ac.pa}

\subjclass[2000]{Primary 16D70} \keywords{Path algebra, Cohn path algebra, Leavitt path algebra, center.}
\begin{document}

\maketitle

\begin{abstract}
This paper is devoted to the study of the center of several types of path algebras associated to a graph $E$ over a field $K$. In a first step we consider the path algebra $KE$ and prove that if the number of vertices  is infinite then the center is zero; otherwise, it is $K$, except when the graph $E$ is a cycle in which case the center is $K[x]$, the polynomial algebra in one indeterminate. Then we compute the centers of prime Cohn and Leavitt path algebras. A lower and an upper bound for the center of a Leavitt path algebra are given by introducing the graded Baer radical for graded algebras.

\end{abstract}

\section{Introduction}

The notion of center plays an essential role in algebra since its very beginning. It appears, for instance, directly related to the zero dimensional cohomology groups of algebras, immediately followed by
the 1-dimensional cohomology groups, that is, derivations. On the other hand, derivations and center are related by the formula $A^-/Z(A)\cong\text{Innder}(A),$ where $A^-$ stands for the Lie algebra (antisimetrization) associated to the associative algebra $A$  and
$\text{Innder}(A)$ is the Lie algebra of inner derivations of $A$. This means that the center can be
considered as the first step in the study of derivations. And motivated in part by the interest on automorphisms and derivations of path algebras, the study of the center, far from being trivial, is boarded in this work. 
A first attempt in the study of the center of a Leavitt path algebra was achieved by Aranda and Crow in \cite{AC}, where they study the center of Leavitt path algebras and get a full description in several settings, for instance, when the Leavitt path algebra is simple. 
\smallskip

Other interesting related problem motivating the study of the center is that of the simplicity of Lie algebras associated to Leavitt path algebras. This idea appears in the paper \cite{AF}, where the authors find conditions which assures the simplicity of the Lie algebra $[M_m(L_K(n))^-, M_m(L_K(n))^-]$, for $L_K(n)$ the Leavitt algebra of type $(1, n)$, equivalently, the Leavitt path algebra of the $n$-petal rose.
\medskip

In spite of this apparently \lq\lq external\rq\rq\ arguments motivating our interest on the center of path algebras we have being also inspired by the general philosophy  of finding the 
characterization of  algebraic properties in terms of properties of the underlying graph. 
Since the begining of the study of Leavitt path algebras, nine years ago, many algebraic properties have being stated at a simple glance of the graph (for example simplicity, primeness, primitivity, existence of a nonzero socle and many others).  In this spirit, we wanted to describe the center of a path algebra / Cohn path algebra / Leavitt path algebra only by looking at its graph.
\medskip

Our initial interest was also the study of the center of graph algebras directly related to Leavitt path algebras, concretely, path algebras and Cohn path algebras. This work turned out to be illuminating for the task of the description of the center of a Leavitt path algebra.
\medskip

The paper is organized as follows. We start Section 2 with some preliminaries and
study the center of the path $K$-algebra $KE$ associated to a connected graph $E$, for $K$ a field; 
the results in this section allow us  to determine the centers of  prime Cohn  and 
Leavitt path algebras in Section 3. To this task, first we reduce the study to the finite case because we obtain that
if a Cohn or Leavitt path algebra has nonzero center then the number of vertices of the underlying graph
must be finite. In the following step we
center our attention in the $0$ component of the center, which turns out to be
finite dimensional, as any element it contains is proved to be symmetric.
Combining this with Theorem \ref{padespues} we show that if $C_K(E)$ is prime, it
must be the Cohn path algebra associated to the $m$-petals rose graph and its center
is $K$ (Subsection 3.3). In Subsection 3.4 we study the center of prime Leavitt path
algebras.  Here the situation is slightly more complex and we get that the center is
$K$ when every cycle in the graph has an exit and the Laurent polynomial algebra
$K[x,x^{-1}]$ in case there is a (necessarily unique) cycle without exits. In
Section 4 we introduce a graded version of the Baer radical of a graded algebra; it
turns out to be zero for any Leavitt path algebra. This result allows to prove that
any Leavitt path algebra $L_K(E)$ is the subdirect product of a family of prime
Leavitt path algebras and to conclude that the center of $L_K(E)$ is a 
subalgebra of a product of centers of prime Leavitt path algebras. We also find  upper and
lower bounds for the center of  $L_K(E)$.

\section{Preliminaries and the center of the path algebra $KE$}
We shall consider always algebras over a base field $K$.
Let us fix some notation and terminology on graphs and algebras. A (\emph{directed}) \emph{graph} $E=(E^0,E^1,r,s)$ consists of two sets $E^0$ and $E^1$ together with maps
$r,s:E^1 \to E^0$. The elements of $E^0$ are called \emph{vertices} and the elements of $E^1$ \emph{edges}.
For $e\in E^1$, the vertices $s(e)$ and $r(e)$ are called the \emph{source} and \emph{range} of $e$,
respectively, and $e$ is said to be an \emph{edge from $s(e)$ to $r(e)$}. If $s^{-1}(v)$ is a finite set for
every $v\in E^0$, then the graph is called \emph{row-finite}. A vertex which emits no edges is called a \emph{sink}; the vertex will be called a \emph{source}  if it does not receive edges. A vertex $v$ is called an \emph{infinite emitter} if $s^{-1}(v)$ is an infinite set, and a \emph{regular vertex} otherwise.
 Also we shall use  the notation $\Path(E)$ for the set of all paths of $E$ including the vertices as trivial paths. For a path $\lambda=e_1 \cdots e_n \in \Path(E)$ we will call the {\emph{length} of $\lambda$ to the number $n$ of edges appearing in $\lambda$ and will denote it by $l(\lambda)$. Vertices are then paths of length $0$.  In this case, $s(\lambda)=s(e_1)$
and $r(\lambda)=r(e_n)$ are the \emph{source} and \emph{range} of $\lambda$. If $\mu$ is a path in $E$, and if $v=s(\mu)=r(\mu)$, then $\mu$ is called a \emph{closed path based at $v$}.
If $s(\mu)=r(\mu)$ and $s(e_i)\neq s(e_j)$ for every $i\neq j$, then $\mu$ is called a \emph{cycle}. An edge $e$ is an {\it exit} for a path $\mu = e_1 \dots e_n$ if there exists $i\in \{1, \dots, n\}$ such that
$s(e)=s(e_i)$ and $e\neq e_i$. Given paths $\alpha, \beta$, we say $\alpha \leq \beta$ if $\beta =\alpha\gamma$ for some path $\gamma$.
\medskip

If $E_1=(E_1^0,E_1^1,s_1,r_1)$ and $E_2=(E_2^0,E_2^1,s_2,r_2)$ are graphs such that 
$E_2^0=E_1^0$, $E_2^1=E_1^1$, $s_2=r_1$ and $r_2=s_1$, then we will say that $E_2$ is the \emph{opposite} graph of $E_1$. We shall use the notation $E_2=E_1^\circ$. It is easy to see that $KE^\circ\cong (KE)^{\text{op}}$, the opposite algebra of $KE$ with multiplication $x\cdot y=yx$. As a consequence there is
a general principle, which we will call {\sl duality}, stating that if $P$ is a property
that holds for the path algebra of any graph, then a \lq \lq dual  property\rq \rq also
holds in the path algebra of any graph.
If $E$ is a graph and $u,v$ two vertices,
we shall say that $u$ and $v$ are \emph{connected} (denoted by $u \sim v$) if they are connected in the underlying undirected graph of $E$ or, in other
words, if there is a (finite) sequence $u=u_1, u_2, \ldots, u_n=v$ such that for any $i$ there is a path $\pi_i$ such that 
$s(\pi_i)=u_i$, $r(\pi_i)=u_{i+1}$ or $s(\pi_i)=u_{i+1}$, $r(\pi_i)=u_{i}$. We shall adopt the convention that any vertex
is connected to itself by a trivial path. So, one can see immediately that {\em connectedness} is an equivalence
relation whose equivalence classes will be called {\em connected components} of $E^0$. 

\medskip

If we write $E^0=\{u_i\}_i$ and 
consider the Peirce decomposition of the path algebra $A=KE$ given  by 
$A=\oplus A_{uv}$, where $A_{uv}:=uAv$, then $A_{uv}=0$ when $u$ and $v$ are not
connected. So if $E^0=\cup_\alpha E^0_\alpha$, where the $E^0_\alpha$ are the different connected components of the set of vertices, for every $\alpha$
we can define the ideals $$A_\alpha=\bigoplus_{u,v\in E^0_\alpha} A_{uv}$$ and we have  $A=\oplus_\alpha A_\alpha$. In this case it
is trivial to check that $A_\alpha A_\beta=0$ when $\alpha\ne\beta$ and $Z(A)=\oplus_\alpha Z(A_\alpha)$, where $Z(\cdot )$ denotes the center of
the corresponding algebra. This observation means that we can restrict our attention to those algebras $KE$ whose graph $E$ is connected.
\medskip

In the mathematical field of graph theory, the distance between two vertices in an undirected graph is the number of 
edges in a shortest path connecting them. 
This is also known as the {\em geodesic distance}  because it is the length of the graph geodesic between those two vertices. 
If there is no path connecting the two vertices, that is to say, if they belong to different connected components, 
then conventionally the distance is defined as infinite.
The vertex set (of an undirected graph) and the distance function form a metric space if and only if the graph is connected.
If $E$ is a directed graph and $u, v \in E^0$ are in the same connected component, we shall define the distance $d(u,v)$ as the geodesic distance
in the underlying undirected graph associated to $E$. If $u$ and $v$ are nonconnected vertices then we shall write $d(u,v)=\infty$.
\medskip

For a subset $S$ of a vector space $V$ we will denote by $\langle S \rangle$ the linear span of $S$ in $V$.

The path algebra $KE$ has a natural  $\Z$-grading whose homogeneous components are $(KE)_n=0$ if $n<0$, $(KE)_0=E^0$ and the $n$-component, $(KE)_n$, for  $n \ge 1$ is:

$$\langle\{\mu \in \Path(E) \colon l(\mu)=n\}\rangle.$$ 

Let $X$ be a set and $T\colon X \to X$ a map, we  define $\hbox{Fix}(T)$ as the set 
$$\hbox{Fix}(T):=\{ x \in X \colon T(x)=x\}.$$

Let $A$ be an algebra and $x \in A$, we can define  the sets $\ran_A(x)=\{y \in A \colon xy=0\}$ and $\lan_A(x)=\{y \in A \colon yx=0\}$. One easy but relevant property of path algebras is the following:
\begin{lm}\label{nuj1}
 Let $A=KE$, $\mu \in \Path(E)$, $u=s(\mu)$  and $v=r(\mu )$:
\begin{enumerate}
\item[\rm (i)]  $\ran_A(\mu)\cap A_{vw}=0$ for all $w \in E^0$.
\item[\rm (ii)]  $\lan_A(\mu)\cap A_{wu}=0$ for all $w \in E^0$.
\end{enumerate}
\end{lm}
\proof{(i). Consider $x\in A_{vw}$ such that $\mu x=0$. 
Write $x=\sum_i k_i\mu_i$, where $k_i\in K$ and $\mu_i$ are paths with source $v$ and range $w$.
We assume that the $\mu_i$'s are all different. 
Then from $ \mu x=0$ we get $\sum_i k_i\mu \mu_i=0$ and since all the paths
in the set $\{\mu \mu_i\}_i$ are different we know that they are linearly independent. Therefore $k_i=0$ for all $i$ and so $x=0$. The second assertion can be proved analogously.}
\medskip

Observe that for  an associative algebra $A$ with a system $\mathcal E$ of
orthogonal idempotents such that
$$A=\bigoplus_{u,v \in \mathcal E} A_{uv}, \ \ \ A_{uv}=uAv$$
the center $Z(A)$ satisfies  $Z(A)\subset \bigoplus_{u \in\mathcal E}A_{uu}$. Thus any central element $z \in Z(A)$ admits a decomposition $z=\sum_{u \in E^0} z_u$, where $z_u=zu \in A_{uu}$. This decomposition will be called the \emph{Peirce decomposition} of $z$ relative to $\mathcal E$.

\begin{lm}\label{conec1}
Let $z\in Z(KE) \setminus \{0\}$ and $u,\ v \in E^0$ such that $u \sim v$, then $zu\ne 0$ if and only if $zv\ne 0$.
\end{lm}
\proof{If there exists  $\mu \in \Path(E)$ such that $s(\mu)=u$ and $r(\mu)=v$, then note that \begin{equation} \label{ec0}zu\mu=z\mu=\mu z=\mu vz.\end{equation} If $zu= 0$ then $\mu zv=0$, hence $zv \in ran_{KE}(\mu)\cap A_{vv}=0$   by Lemma \ref{nuj1} (i). If $zv=0$ then $zu\mu=0$, that is, $zu \in lan_{KE}(\mu)\cap A_{uu}=0$ by Lemma \ref{nuj1} (ii).
Next we proof the general case: there is a finite sequence $u=u_0,\ u_1,\ldots , u_n=v$ such that for any $i$ there is an $f \in E^1$ such that $s(f)=u_i$ and $r(f)=u_{i+1}$ or  $s(f)=u_{i+1}$  and $r(f)=u_{i}$. Then $zu_i\ne 0$ if and only if $zu_{i+1}\ne 0$; this proves the lemma.}

\medskip
\begin{co}\label{topo}
If $E$ is a connected graph and $Z(KE)\ne 0$, then $\vert E^0\vert$ is finite.
\end{co}
\proof{Assume $z=\sum_{u\in E^0} zu \in Z(KE) \setminus \{0\}$, where only a finite number of summands is nonzero.  Take a vertex $u$ such that $zu\ne 0$. For any vertex $v \in E^0$ we have $u \sim v$ and applying Lemma \ref{conec1} we get  $zv\ne 0$. Since only a finite number of summands is non zero, then $E^0$ must be finite.}
\medskip

\subsection{The center of the path algebra $KE$.}

From now on we shall assume that $E$ has a finite number of vertices. Moreover, if $E$ is not connected it must be a finite union of finite connected graphs $E_i$ and $KE$ is a direct sum of the algebras $KE_i$. Furthermore, $Z(KE)$ is the direct sum of the centers $Z(KE_i)$. Hence we can focus our attention on finite connected graphs.

\begin{lm}\label{lemi}
Let $z \in Z(KE)$ with  Peirce decomposition  $z=\sum_{u \in E^0} z_u$. If $f\in E^1$ is such that  $s(f)=u, r(f)=v$, then $z_u \in Ku$ if and only if $z_v \in Kv$.
\end{lm}
\proof{Assume $z_u=ku$; since $z_uf=fz_v$ we have $kf=fz_v$. If $k=0$ then $z_u=0$ and by Lemma \ref{conec1} we get $z_v=0$. In case $k\ne 0$ one has $\deg (kf)=1=\deg (fz_v)$ therefore $\deg (z_v)=0$ and $z_v \in Kv$.
Now, if $z_v=hv$ the same ideas lead one to the conclusion that $z_u \in Ku$.}
\begin{pr}\label{virtue}
If $u \sim v$ then $z_u \in Ku$ if and only if  $z_v \in Kv$.
\end{pr}
\proof{We proceed by induction on the geodesic distance $n$ between $u$ and $v$. For $n=1$ apply Lemma \ref{lemi}. If $n>1$ there are vertices $u=u_1, u_2, \ldots, u_n=v$ such that for any $i \in \{1,\ldots, n\}$ there is an arrow from $u_i$ to $u_{i+1}$ or  from $u_{i+1}$ to $u_i$. So the induction hypothesis implies  $z_{u_{n-1}} \in Ku_{n-1}$ and then by Lemma \ref{lemi},  $z_{u_n}=z_v\in Kv$.}
\medskip

As a consequence of the previous result, for a finite connected graph $E$, if $z \in Z(KE)$ and its Peirce decomposition is $z=\sum_{u \in E^0}z_u$,  we have the following dichotomy:
\begin{enumerate}
\item[\rm (i)]  $z_u \in Ku$ for all $u \in E^0$ .
\item[\rm (ii)]  $z_u \notin Ku$ for all $u \in E^0$.
\end{enumerate}

In the first case the central element is of the form $z=\sum_{u \in E^0}k_u u$, where $k_u \in K$. Next we  prove that all the scalars $k_u$ agree. 

\begin{lm}\label{nuj0}
 Let $z \in Z(KE)\setminus \{0\}$ with Peirce decomposition of the form $z=k u+h v+\sum_{w\ne u,v}z_w$, where $u\ne v$, $k,h\in K^\times$ 
and $z_w\in A_{ww}$. If $u$ and $v$ are connected, then $k = h$.
\end{lm}
\proof{Proceed by induction on the number $d=d(u,v)$. If $d=1$ 
we may assume without loss in  generality that there is an $f\in E^1$ such that $s(f)=u$ and $r(f)=v$. 
Then $z f=f z$ yields $k f=h f$, hence $k=h$.
Now suppose that the property holds whenever $d<n$. Consider now two vertices $u$ and $v$ such that  $d(u,v)=n$. There is a vertex $w$
with $d(u,w)=1$ and $d(w,v)=n-1$. Then  there exists $f\in E^1$ with either $s(f)=u$, $r(f)=w$ or $s(f)=w$, $r(f)=u$.
In the first case, $fz=zf$ implies  $fz_w=k f$ hence $f (z_w-k w)=0$ and $z_w-k w\in\ran_{KE}(f)\cap A_{ww}=0$  by Lemma \ref{nuj1} (i), so $z_w=kw$. 
Now $z=k u+h v+k w+\sum_{u'\ne u,v,w}z_{u'}$ and applying the induction hypothesis to $v$ and $w$ we get $k=h$. 
In the second case the proof is similar by using Lemma \ref{nuj1} (ii).
}
\medskip

Thus for a connected finite graph $E$, the central elements are of the form (i) $z=k \sum_{u \in E^0}u=k1$, where $k \in K$, or (ii) $z= \sum_{u \in E^0}z_u$, where $z_u \notin Ku$ for all $u \in E^0$. The elements of the form $k1$ will be called \emph{scalars elements}. From now on, we shall investigate under which conditions the path algebra $KE$ has nonscalar central elements.

\begin{de}\rm
Let $S$ denote the set of all nontrivial paths for a graph $E$. We define the map $F_e\colon S \to E^1$ given by $F_e(f_1\ldots f_n)=f_1$. We shall call this map the \lq\lq\emph{first edge}\rq\rq  \ map.
\end{de}

\begin{lm}\label{unicaflecha}
Let  $z$ be a nonscalar central element with Peirce decomposition $z=\sum_{v\in E^0} z_v$. If   $0\ne z_u= \sum_{i \in I} k_i\lambda_i$ with $k_i \in K^{\times}, \lambda_i \in \Path(E)$  and $f   \in E^1\cap s^{-1}(u)$ then $f=F_e(\lambda_i)$ for each $i \in I$.
 \end{lm}
\proof{Let $v:=r(f)$ and write $z_u=\sum_{i \in I} k_i\lambda_i$ and $z_v=\sum_{j \in J} h_j \mu_j$, with $k_i,h_j \in K^{\times}, \lambda_i, \mu_j \in \Path(E)$. Observe that $\lambda_i$ and $\mu_j$ are nontrivial paths because $z$ is nonscalar and by virtue of Proposition \ref{virtue}. Since $fz=zf, we get z_uf=fz_v$, that is,
\smallskip
\begin{equation}\label{ecuotra}\sum k_i\lambda_if-\sum h_j f\mu_j=0.\end{equation}
\smallskip
We claim that $\{\lambda_i f\}\cap\{f\mu_j \}\ne \emptyset$ because, otherwise, $\{\lambda_i f\}\cup\{f\mu_j  \}$ would be linearly independent, so $k_i=0=h_j$ for all $i\in I,j\in J$ and therefore $z_u=0=z_v$, a contradiction. Next we prove that for every $i \in I$ there exists a unique $j \in J $ such that $\lambda_i f=f\mu_j $. Assume on the contrary  that there exists $i_0 \in I$ such that $\lambda_{i_0}f\notin \{f\mu_j\}$. Then rewrite (\ref{ecuotra}) to get $k_{i_0}\lambda_{i_0}f+ \sum_{i \ne i_0}k_i\lambda_if-\sum h_jf\mu_j=0$. This implies $k_{i_0}=0$, a contradiction. Now from $\lambda_i f=f\mu_j $ we have $F_e(\lambda_i)=f$ for each $i \in I$.}

\medskip

\begin{re}\label{jun0}
 {\rm Lemma \ref{unicaflecha} implies in particular that for a finite connected graph $E$ with nonscalar center there are no bifurcations at any $u \in E^0$.}
\end{re}

\medskip

\begin{lm}\label{nsns}
If $0\ne Z(KE) \not \subset K\cdot 1$ then there are no sinks and no sources in the graph $E$.
\end{lm}
\proof{By duality, it suffices to prove that there are no sinks. Assume that $v \in E^0$ is a sink. Let $z \in Z(KE)\setminus K\cdot 1$ be with Peirce decomposition $z=\sum_{u \in E^0} z_u$; then, by Proposition \ref{virtue}, each $z_u \not \in K u$.  In particular $z_v \not \in K v$. Hence $z_v=\sum_{i\in I} k_i \mu_i$ with $s(\mu_i)=r(\mu_i)=v$, being $\mu_i$  nontrivial, which is a contradiction because $s^{-1}(v)=\emptyset$.}

\begin{pr}
If $0\ne Z(KE) \not \subset K\cdot 1$ then $E$ is a cycle.
\end{pr}
\proof{We know that $E$ is finite, connected, with no sinks and no sources (by Lemma \ref{nsns}) and without bifurcations (see Remark \ref{jun0}). Then it is easy to prove that $E$ is a cycle; by Corollary \ref{topo} we may assume  $\vert E^0 \vert =n \in \N\setminus\{0\}$. The proof is clear for $n=1$ so  assume $n>1$.  Take any vertex  $u_1$; since it is not a sink there is a unique edge $f_1 \in s^{-1}(u_1)$;  define $u_2=r(f_1)$. If $u_1, \cdots ,u_{i-1}$ have been defined for $i<n$ let $u_i=r(f_{i-1})$, where $f_{i-1}$ is the unique edge in $s^{-1}(u_{i-1})$. Next we proof that $u_i \not \in \{u_1, \ldots ,u_{i-1}\}$. Suppose on the contrary $u_i=u_q$ for $1\le q\le i-1$; then $u_{i+1} \in \{u_1, \ldots ,u_{i-1}\}$. Thus $E^0=\{u_1, \ldots ,u_{i-1}\}$, which contradicts the fact that $i\le n$. We conclude that $E^0=\{u_1, \ldots ,u_{n}\}$ and as $u_n$ is not a sink there is only one edge $f_n \in s^{-1}(u_n)$. Now if $r(f_n)=u_i$, with $i\ge 1$, then $u_1$ is a source, a contradiction. Consequently $r(f_n)=u_1$ and $E$ is a cycle.}

\begin{pr}
If $E$ is a cycle then $Z(KE)\cong K[x]$, the polynomial algebra in the indeterminate $x$. More precisely, if $E^0=\{u_1, \ldots ,u_{n}\}$ and $E^1=\{f_1, \ldots ,f_{n}\}$, with $s(f_i)=u_i$ for every $i$, $r(f_i)=u_{i+1}$ for $i=1, \ldots ,n-1$, and $r(f_n)=u_1$, let 
$$  \begin{matrix} c_1 & = & f_1\cdots f_n\cr
c_2&= & f_2 \cdots f_nf_1\cr
\vdots &   & \cr
c_i&=&f_i \cdots f_nf_1 \cdots f_{i-1}.
\end{matrix}$$
Then 
$$Z(KE)=\left\{ \sum_{i=1}^np(c_i)\colon p(x) \in K[x]\right\}$$
and there is an isomorphism from $Z(KE)$ to $K[x]$ such that  $\sum_{i=1}^np(c_i) \mapsto p$.
\end{pr}
\medskip

We collect the results and remarks above in the following

\begin{theo}\label{teoro}
 If $E$ is a graph then $Z(KE)$ is the direct sum of the centers of the path algebras associated to the connected components
of the graph. If $E$ is connected and has an infinite number of vertices then $Z(KE)=0$. If $E$ is connected and has a finite number
of vertices then
$Z(KE)=K {\cdot} 1$ except if $E$ is a cycle; in this case  $Z(KE)\cong K[x]$.
\end{theo}
\medskip

\subsection{Relationship with the center of other classes of algebras.}
Given a graph $E$ we can define the \emph{extended graph} of $E$ as the new graph $\hat{E}=(E^0,E^1\cup (E^1)^*, r', s')$, where $ (E^1)^*=\{e^* \colon e \in E\}$ and the maps $r'$ and $s'$ are defined as $r'\vert_{E^1}=r$, $s'\vert_{E^1}=s$, $r'(e^*)=s(e)$ and $s'(e^*)=r(e)$.
In this subsection we would like to establish some relationships between the center of $K\hat{E}$ and the center of certain types of algebras related to the path algebra $K\hat{E}$.

\medskip 

Consider the following sets of elements in $K\hat E$:
\medskip

\begin{enumerate}    
\item[(CK1)] $e^*e'-\delta _{e,e'}r(e) \ \mbox{ for all } e,e'\in E^1$.
\item[(CK2)] $v-\sum _{\{ e\in E^1\mid s(e)=v \}}ee^* \ \ \mbox{ for every regular
vertex }v\in E^0.$
\end{enumerate}
\medskip

Then the {\it Leavitt path} $K$-{\it algebra} associated to the graph $E$, denoted
$L_K(E)$,  can be described as $L_K(E)=K\hat{E}/J_1$, where 
$J_1$ is  the ideal generated by the elements in (CK1) and (CK2) (see, for example
\cite{AAS}).
\medskip

In the book \cite{AAS} the authors define the {\it Cohn path} $K$-{\it algebra}
associated to $E$, denoted by $C_K(E)$, as $K\hat{E}/J_2$, where 
$J_2$ is the ideal generated by the elements in (CK1).
\medskip

In both cases the described path algebras arise as quotients of the path algebra
over the extended graph module an ideal, and in order to determine their centers, we
may follow a similar scheme.
\medskip

 So, consider an ideal $I$ of $K\hat{E}$ and the algebra $A:=K\hat{E}/I$. Then there is a short  exact sequence

\begin{equation}\label{sequence}\xymatrix{ 
0\ar[r]&I\ar[r]^j& K\hat{E}\ar[r]^p&A\ar[r]&0}\end{equation}
\smallskip

\noindent
where $j$ is the inclusion and $p$ the canonical projection. Thus,  for $I=J_1$ the algebra $A$ is isomorphic to $L_K(E)$ and for $I=J_2$  it is isomorphic to $C_K(E)$.  Observe that  $J_i\cap E^1= \emptyset$ for $i=1,2$.

\begin{pr}
 Let $E$ be a connected graph and $I$ an ideal of $K\hat E$ such that $I\cap E^1= \emptyset$. Define $A:=K\hat E /I$ and consider the short exact sequence in (\ref{sequence}). If $E^0$ is finite then $\lin{p(E^0)}\cap Z(A)=K.1$; otherwise, $\lin{p(E^0)}\cap Z(A)=0$.
\end{pr}

\proof{Take $z \in \lin{p(E^0)} \cap Z(A)$. Then $z=\sum_{k \in S} l_k p(u_k)$, with $l_k\in K$, $u_k \in E^0$ and where $S$ can be infinite but only a finite number of the $l_k$'s  are nonzero. Let  $i\ne j$; if the geodesic distance $d(u_i,u_j)=1$ then there is an $f \in E^1$ with $s(f)=u_i$ and $r(f)=u_j$ (if necessary swap $i$ and $j$). Since $p(f)z=zp(f)$ we have $p(f)z=p(f)\sum_{k \in S} l_k p(u_k)=p(\sum_{k \in S} l_k fu_k)=l_jp(f)$ and $zp(f)=\sum_{k \in S} l_k p(u_k)p(f)=l_ip(f)$. Hence $(l_i-l_j)p(f)=0$ and $(l_i-l_j)f \in I=\hbox{Ker}(p)$. If $l_i \ne l_j$ then $f \in I\cap E^1=\emptyset$, a contradiction. Thus $l_i = l_j$. Suppose now that the geodesic distance $d(u_i,u_j)=n>1$; then there exists a vertex $u_k$ with  $d(u_i,u_k)<n$ and $d(u_j,u_k)<n$. Applying a suitable induction hypothesis we have $l_i=l_k=l_j$. As a consequence, if $E^0$ is finite then $z$ is a multiple of the unit and if $E^0$ is 
infinite  some scalar  $l_k $ must be zero hence all of them are zero; therefore $z=0$.}

\begin{co}
If $A$ is the Cohn path algebra or the Leavitt path algebra of a connected graph then  
$$\lin{E^0}\cap Z(A)=\begin{cases}K.1& \text{if } E^0 \text{ is finite} \cr
0 &\text{otherwise.}\end{cases}$$
\end{co}

\section{Centers of prime Cohn and Leavitt  path algebras}
Once we have determined the center of $KE$ we are interested in the study of the centers of the Cohn path algebra $C_K(E)$ and of the Leavitt path algebra $L_K(E)$. 
As it has been said, a first step in the study of the center of a Leavitt path algebra was given in  \cite{AC}, where the authors determined the center of a simple Leavitt path algebra. The following natural step is to try to determine the center of  prime Leavitt path algebras. 

\medskip

The starting point in this section will be to prove that the existence of a  nonzero center for a prime Leavitt or Cohn path algebra forces the finiteness of the number of vertices in the graph. One of the relevant tools in the theory of Cohn and of Leavitt path algebras which we shall need is the natural $\Z-$grading,  where the 
vertices have degree $0$ and the elements of the form $\sigma\tau^*$ have degree $n-m$ for $\sigma$ and $\tau$  paths of lengths $n$ and $m$ respectively. The degree of a homogeneous element $x$ in a Cohn or Leavitt path algebra will be denoted by $\deg(x)$. 
\subsection{ Reduction to the finite case.}
The graphs $E$ in this subsection are not necessarily row-finite unless otherwise specified.
Recall that for a graph $E$ a subset $H\subset E^0$ is said to be \emph{hereditary} when for any two vertices $u,v$ such that there is a path $\mu$ with
$s(\mu)=u$ and $r(\mu)=v$, if $u\in H$ then $v\in H$. A
hereditary set is \emph{saturated} if every regular vertex which feeds into $H$ and only into $H$ is again
in $H$, that is, if $s^{-1}(v)\neq \emptyset$ is finite and $r(s^{-1}(v))\subseteq H$ imply $v\in H$. We recall also that the ideal of the Cohn path algebra or of the Leavitt path algebra $A=C_K(E)$ or $L_K(E)$ generated by a hereditary set $H$ is easily seen to agree with the set of all linear combinations of elements of the form $\alpha\beta^*$, where $\alpha$ and $\beta$ are paths such that $r(\alpha)\in H$.
\medskip

\begin{lm}
If $A=C_K(E)$ or $L_K(E)$ and $\mu \in \Path(E)$ with $v=r(\mu)$, then the left multiplication operator $L_{\mu} \colon A \to A$ given by $a  \mapsto \mu a$ satisfies  $\ker(L_{\mu})\subset \ran_A(v)$. Moreover, $\ker(L_{\mu}) \cap A_{vw}=0$ for all $w \in E^0$.
\end{lm}
\proof{If $L_{\mu}(a)=0$ then $\mu a =0$; therefore $\mu^*\mu a=0$, that is, $va=0$ and $a \in \ran_A(v)$. If $a\in A_{vw} \cap \ker(L_{\mu})$ then $a \in \ran_A(v)$; hence $va=a \ (a \in A_{vw})$ or $va=0 \ (a \in \ran_A(v))$; in both cases $a =0$.}
\begin{pr} Let $E$ be any graph and $A=C_K(E)$ or $L_K(E)$. Let $z\in Z(A)\setminus\{0\}$ and consider its Peirce decomposition $z=\sum_w z_w$.
Fix $u,v\in E^0$ and assume that there is some path $\mu$ with $s(\mu)=u$
and $r(\mu)=v$; then $z_u=0$ implies $z_v=0$.
\end{pr}

\proof{Since $z\mu=\mu z$ we get $z_u\mu=\mu z_v$.  If $z_u=0$ we have
$0=\mu z_v$, that is, $z_v\in\ker(L_\mu)\cap A_{vv}=0$.}

\begin{co} Under the conditions in the previous proposition the set
$$H:=\{u\in E^0\colon z_u=0\}$$ is hereditary.
\end{co}

\begin{co}\label{cuatro}Under the conditions in the  proposition above we have
$Az I(H)=0$.
\end{co}

\proof{Take $a\in A$ and $q\in I(H)$. Then $q=\sum_i l_i \alpha_i\beta_i^*$, with $l_i\in K$ and $r(\alpha_i)\in H$ for each $i$. Then $azq=a\sum_i l_i \alpha_i z\beta_i^*=
a\sum_i l_i\alpha_i z r(\alpha_i)\beta_i^*=0$ since $r(\alpha_i)\in H$.}

\begin{pr}\label{finito}If $E$ is a graph and $A=C_K(E)$ or $L_K(E)$ is prime, then $Z(A)\ne 0$ implies that $E^0$ is finite.
\end{pr}

\proof{Take a nonzero central element $z$. Since $A$ is prime and $Az$, $I(H)$ are ideals of $A$ whose product is zero by Corollary \ref{cuatro}, we conclude  that either $Az=0$ or $I(H)=0$. But since $z\ne 0$ we have $Az\ne 0$. Thus
$I(H)=0$, hence $H=\emptyset$ and we conclude that $z_u\ne 0$ for any $u$.
This forces the finiteness of $E^0$ since the number of nonzero components
in the Peirce decomposition of $z$ is necessarily finite.}
\medskip

Now we know that in order to have  nonzero center for the prime algebras $A=C_K(E)$ or $L_K(E)$ we need finiteness of $E^0$.
 
\begin{re}\rm It is easy to realize that for any positive integer $n$ there is a graph $E$ with $\vert E^0\vert=n$ such that $L_K(E)$ is prime and with nonzero center. Indeed, consider the graph consisting of $n$ vertices arranged in a single line. This is  a simple algebra $L_K(E)$ whose center is $K$. However, we will prove that  for a prime Cohn path algebra $C_K(E)$ the number of vertices  in $E$ must be $1$. 
\end{re}

Let $E$ be a row-finite graph and $A:=C_K(E)$.  Let $u\in E^0$ and suppose $u$ is not a sink. Write $s^{-1}(u)=\{f_i\colon i=1,\ldots,n\}$. Then  for any nontrivial
path $\mu$ we have 

\begin{equation}\label{ick1}(u-\sum_i f_if_i^*)\mu=0 \ \text{or, equivalently,} \ \mu^*(u-\sum_i f_if_i^*)=0.
\end{equation} 

The element $u-\sum_i f_if_i^*$ is an idempotent and 

\begin{equation}\label{fb1} (u-\sum_i f_if_i^*)A(u-\sum_i f_if_i^*)=K(u-\sum_i f_if_i^*). 
\end{equation}

This follows by (\ref{ick1}) and taking into account $(u-\sum_if_if_i^*)v=\delta_{u,v}(u-\sum_if_if_i^*)$.
Moreover, if $v\in E^0$ is a sink then 

\begin{equation}\label{kaw2}
(u-\sum_i f_if_i^*)Av=vA(u-\sum_i f_if_i^*)=0
\end{equation} 

\noindent
and for any two different vertices $u,v\in E^0$ (neither of them being  sinks) we have

\begin{equation}\label{kaw1}
(u-\sum_i f_if_i^*)A(v-\sum_j g_ig_i^*)=0,
\end{equation}

\noindent
where $s^{-1}(u)=\{f_i\}$ and $s^{-1}(v)=\{g_j\}$. 
\medskip

At this point it would be convenient to recall the elementary characterization of primeness for an associative algebra:
an algebra $A$ is said to be \emph{prime} if and only if for any two elements $a,b\in A$ the fact $aAb=0$ implies $a=0$ or $b=0$. \medskip

At the level of Leavitt path algebras there is a purely graph-theoretic characterization of primeness. Recall that a graph $E$ satisfies \emph{Condition} (MT3) if for every $v,w\in E^0$ there exist $u\in E^0$  and paths $\mu,\tau\in \Path(E)$ such that
$s(\mu)=v$, $s(\tau)=w$ and $r(\mu)=r(\tau)=u$. In recent papers on prime
Leavitt path algebras, the term \lq\lq Condition (MT3)\rq\rq   \ has been replaced by
the more descriptive term \emph{downward directed}. It is proved in \cite{APS2} that in the row-finite case, $L_K(E)$ is prime if and only if the graph $E$ is downward directed. 

\begin{theo}\label{cohnprima}
Let $E$ be a row-finite graph; then the Cohn path algebra $C_K(E)$ is prime if and only if  $\vert E^0\vert =1$.
\end{theo}
\proof{Assume first that $C_K(E)$ is prime. Take $u,v\in E^0$ different. If they
are not sinks then formula (\ref{kaw1}) contradicts the primeness of $A$. If one of them is a sink, then formula
(\ref{kaw2}) contradicts also the primeness of the algebra. Finally, if $u$ and $v$ are sinks we have $uAv=0$  
contradicting once more the primeness of the algebra. Hence $\vert E^0\vert=1$. Then using \cite[Theorem 1.5.17]{AAS} with $X=\emptyset$ the Cohn path algebra $C_K(E)$ is isomorphic to the Leavit path algebra $L_K(F)$ which is prime because $F$ is downward directed (see the graphs below).
$$E : \ \ \ \xymatrixcolsep{3.3pc} \xymatrix@-1.0pc{\bullet^{\emph{v}}\ar@(l,u)@{.}[]\ar@(ul,ur)[]^{f_{n}} \ar@(u,r)[]^{f_{1}}\ar@(ur,dr)[]^{f_{2}}\ar@(r,d)[]^{f_{3}}\ar@(dr,dl)@{.}\ar@(l,d)@{.}\ar@(dl,ul)@{.}} \ \ \ \ \ \ \ \ F :  \\ \xymatrix@-1.0pc{ {\bullet}^{v'}& \bullet^{\emph{v}}\ar@(l,u)@{.}[]\ar@(ul,ur)[]^{f_{n}} \ar@(u,r)[]^{f_{1}}\ar@(ur,dr)[]^{f_{2}}\ar@(r,d)[]^{f_{3}}\ar@(dr,dl)@{.}\ar@(l,d)@{.} \ar@{>} [l]_{(n)} }$$
}

\subsection{Some properties of the centers.} 

In this subsection we shall study properties of the centers of Cohn path
algebras and Leavitt path algebras that will be of interest for our purposes.

\begin{theo}\label{padespues}
Let $A$ be a $\Z$-graded algebra, $A=\oplus A_n$, with an involution $\ast$ such that
$A_n^*=A_{-n}$ for any integer $n$.
Let $Z:=Z(A)$ be its center provided with the induced
$\Z$-grading, i.e.,  $Z=\oplus_n Z_n$, where $Z_n=Z\cap A_n$. Suppose that $Z$ is a domain and $Z_0$ a field.
Then  each component $Z_n$ is isomorphic (as a $Z_0$-vector space) to $Z_0$.
Moreover either $Z=Z_0$  or there is an isomorphism $Z\cong Z_0[x,x^{-1}]$ of graded $F$-algebras.
\end{theo}

\proof{If $Z_n\ne 0$ take $0\ne x\in Z_n$. Then $xx^*\in Z_0$. Moreover since $Z$ is a domain $0\ne xx^*\in Z_0$ hence
$x$ is invertible. Thus any nonzero homogeneous element is invertible. 
\newline
Consider now the linear map $L_x\colon Z_0\to Z_n$ such that $L_x(a)=xa$. This is bijective since $x$ is invertible
and so it is an isomorphism of $Z_0$ vector spaces.  If there is no positive integer $n$ such that $Z_n\ne 0$ then $Z=Z_0$.
Suppose on the contrary this is not true and consider the minimum positive integer $n$ such that
$Z_n\ne 0$.  Then $Z=\oplus_{i\in \Z}Z_{in}$ because  if there exists $k \in \Z$ with $in<k<(i+1)n$ and $Z_k\ne 0$ then $Z_kZ_{in}^*\subset Z_{k-in}=0$ so $Z_k=0$. As $Z_n=Z_0x$, then $Z_{in}=Z_0x^i$ and $Z=\oplus_{i\in \Z}Z_{in}=\oplus_{i\in \Z}Z_0x^i$, so $\{x^i \colon i \in \Z\}$ is a basis of $Z$ as a  $Z_0$-vector space. Therefore $Z\cong Z_0[x,x^{-1}]$.
}

\medskip
Recall that an algebra $A$ is \emph{graded simple} if and only if $A^2\ne 0$ and its only graded ideals are $0$ and $A$.

\begin{co}
Let $A=L_K(E)$ or $A=C_K(E)$ be a graded simple algebra. If $Z:=Z(A)\ne 0$ then $Z_0$ is a field and $Z=Z_0$ or $Z=Z_0[x,x^{-1}]$ (the second possibility does not occur if $A$ is simple).
\end{co}

\proof{Any nonzero homogeneous element in the center is invertible (because the ideal generated by this element in $A$ is graded and nonzero). In particular, $Z_0$ is a field. Let us prove now that $Z$ is a domain. Take $x,y \in Z$ such that $xy=0$ but $x,y \ne 0$. Write $x=\sum_{i\in \Z}x_i$, with $\ x_i \in Z\cap A_i$ and  $y=\sum_{i\in \Z}y_i$, where $\ y_i \in Z\cap A_i$; consider $n:=\max\{i \colon x_i\ne 0\}$ and $m:=\max\{i \colon y_i\ne 0\}$. Since $xy=0$ we have  $x_ny_m=0$, which is impossible because nonzero homogeneous elements are invertible. }
\medskip

\begin{de}\rm
Let $A$ be the Cohn path algebra $C_K(E)$ or the Leavitt path algebra $L_K(E)$ associated to a graph $E$. Let $z \in A_0$. We say that $z$ is \emph{symmetric} if it is  a linear combination of elements of the form $\mu\mu^*,$ where $\mu \in \Path(E)$.
\end{de}

In what follows we will see that for  $A=L_K(E)$ or $C_K(E)$ and $Z=Z(A)$, any element in the zero component of the center is symmetric, i.e., is invariant under the involution. 

\medskip

If $z\in Z_0$ and we write $z$ as a linear
combination of linearly independent monomials (see \cite[Proposition 1.5.6 and Corollary 1.5.11]{AAS}), then
every nonzero monomial in $z_0$ is of the form $f\mu\tau^*f^*$, where $f\in E^1$ or
it is a scalar multiple of a vertex.  
To prove this statement, write $z=\sum_i l_i f\mu_i\tau_i^*g^*+r$, for $l_i\in K^\times$ (where $K^\times = K \setminus \{0\})$,
$f\ne g$, $f,g\in E^1$, $\mu_i,\tau_i\in \Path(E)$,
with $\deg(\mu_i)=\deg(\tau_i)$ and $r$ stands for the remaining summands in $z$ which
either do not start by $f$ or do not end by $g^*$ (so
we have $f^*rg=0$). Then
$$0=zf^*g=f^*zg=\sum_i l_i \mu_i\tau_i^*\hbox{ implying } 0=\sum_i l_i
f\mu_i\tau_i^* g^*$$
and since the monomials $\{f\mu_i\tau_i^*g^*\}$ are linearly independent we get
$l_i=0$, contradicting the fact that $l_i\in K^\times$.
This proves that central elements of degree zero are linear combinations of vertices
and monomials of the form $f\mu\tau^*f^*$, with
$\mu,\tau\in \Path(E)$. \medskip
\medskip

The set $\{\sigma \mu^* \colon \sigma , \mu \in \Path(E) \}$ is a basis for any Cohn path algebra (see \cite[Proposition 1.5.6]{AAS}). For each of these basic elements we define the \emph{real degree} of $\sigma\mu^*$ (denoted by $\gr(\sigma\mu^*)$) as the length of $\sigma$.
\medskip

\begin{de}\rm
For any $ \omega \in C_K(E)$ define $\gr(\omega)= \max\{\gr(\sigma_i\mu_i^*)\}$, where $\omega=\sum_{i\in I} l_i\sigma_i\mu_i^*$ is the expression of $\omega$ as a linear combination of the elements in the mentioned basis of the algebra.
\end{de}

We can also define the notion of real degree in the context of Leavitt path algebras by considering the basis given in \cite[Corollary 1.5.11]{AAS}. Since each element in this basis is of the form $\lambda\mu^*$ we can define the \emph{real degree} of this element as the length $l(\lambda )$. Then the \emph{real degree} of an arbitrary element can be defined as the maximum of the real degrees of the basis elements in its expression as a linear combination of the basis.
\medskip

Next we proceed to prove that any homogeneous element $z$ of degree zero in the center is
symmetric. 
\medskip

Assume that 
$z=\sum_i l_i \mu f\alpha_i\beta_i^*g^*\mu^*+s+r$, where $l_i\in K$,
$\mu,\alpha_i,\beta_i\in  \Path(E)$, $f,g\in E^1$, $f\ne g$,
$\deg(\alpha_i)=\deg(\beta_i)\ge 0$, $s$ is symmetric being all its summands 
of real degree $\le n$, and $r$ is a linear combination of walks whose real degrees
are $>\deg(\mu)$ and
that either do not start with $\mu f$ or do not
end with $g^*\mu^*$. Observe that we can assume without loss of generality that the
set of walks $\{\mu f\alpha_i\beta_i^*g^*\mu^* \}$ is linearly independent. \medskip

\begin{lm} We have  $f^*\mu^*r\mu g=0$.
\end{lm}

\proof{Take a summand of $r$, which we know is of the form $g_1\cdots
g_qh_q^*\cdots h_1^*$, with $q>n=\deg(\mu)$.
Then, if $\mu=t_1\cdots t_n$ we have 
$$f^*t_n^*\cdots t_1^*g_1\cdots g_qh_q^*\cdots
h_1^*t_1\cdots t_n g\ne 0$$
if and only if $t_i=g_i=h_i$ for $i=1,\ldots, n$ and $g_{n+1}=f$, $h_{n+1}=g$. But
this implies that the summand $g_1\cdots g_qh_q^*\cdots h_1^*$ starts
by $\mu f$ and ends by $g^*\mu^*$, a contradiction.}
\medskip

\begin{lm} We get $f^*\mu^*s\mu g=0$.
\end{lm}

\proof{Consider a summand of $s$, say $g_1\cdots
g_qg_q^*\cdots g_1^*$, with $q\le n$. 
Then, in case $n>q$ and since $\mu=t_1\cdots t_n$, we have 
$$f^*t_n^*\cdots t_1^*g_1\cdots
g_q g_q^*\cdots g_1^*t_1\cdots t_n g= 
\Pi_{i=1}^q \delta_{g_i,t_i}f^*t_n^*\cdots t_{q+1}^*t_{q+1}\cdots
t_ng=$$ 
$$ \Pi_{i=1}^q \delta_{g_i,t_i} f^*g=0.$$
And if $n=q$ we get similarly $f^*t_n^*\cdots t_1^*g_1\cdots g_q g_q^*\cdots
g_1^*t_1\cdots t_n g=0$.}
\medskip

\begin{theo} Every element of degree zero in the center of $C_K(E)$ or of $L_K(E)$ is symmetric.
\end{theo}

\proof{Take $z  \in Z_0$ as before. We have $0=f^*\mu^*\mu g z=f^*\mu^*z\mu g=\sum_i l_i \alpha_i\beta_i^*$. But
the set $\{\alpha_i\beta_i^*\}$ is linearly independent 
because $\{\mu f\alpha_i\beta_i^*g^*\mu^* \}$ is. This implies $l_i=0$ and therefore
$z$ is symmetric.}

\font\hu=msbm10
\def\N{\hbox{\hu N}}
\def\im{\mathop{Im}}

\begin{de}\rm
Let $A$ be an algebra with an involution $\ast$ and $a$ an element in $ A$. Then we define the  operator $T_a$ as the linear map $T_a\colon A \to A$ given by $T_a(x):=a^*xa$.
\end{de}
Note that for any $a, \ b \in A$ we have $T_aT_b=T_{ba}$.

\begin{lm}Let $A$ be $L_K(E)$ or $C_K(E)$ and let $f\in E^1$ be such that
$s(f)=r(f)=u$. Consider the linear map $T_f \colon uAu\to uAu$. Then,  if $\omega$ is a fixed point of $T_f$ of degree zero, we
have $\omega=k u$ for some
scalar $k$.
\end{lm}
\proof{Assume $T_f(\omega)=\omega$, then $T_f^n(\omega)=\omega$ for all $n$. 
It is easy to see that any element $h_1\cdots h_k g_1^*\cdots g_k^*$ with $h_i, g_i \in E^1$ and some $h_i\ne
f$ or some
$g_j\ne f$ is in the kernel of $T_f^m$ for suitable $m$. Thus $\omega$ is a linear
combination
of elements of the form $f^n(f^*)^m$. But since $\deg(\omega)=0$ we have $n=m$.
Now $T_f(f^n(f^*)^n)=f^{n-1}(f^*)^{n-1}$  implies that $\omega$ is a scalar
multiple of $u$.}

\begin{co} Let $E$ be a graph. If $f\in E^1$ is such that
$s(f)=r(f)=u$, take $z\in Z_0=Z(A)_0$, for $A=L_K(E)$ or $C_K(E)$.
If $z=\sum z_w$  is the Peirce decomposition of $z$ then $z_u\in
Ku$.\end{co}
\proof{Since $zf=fz$ and $s(f)=r(f)=u$ we have $z_
uf=fz_u$. On the other hand we have
$T_f(z_u)=f^*z_uf=f^*fz_u=z_u$ and by the previous lemma we conclude $z_u\in Ku$.}

\begin{lm} Let $A$ be $L_K(E)$ or $C_K(E)$ and let $c=e_1\dots e_n \in \Path(E)$ be a closed path based at $u$. Consider the linear map $T_c \colon uAu\to uAu$ given by
$T_c(x):=c^*xc$. If $\omega$ is a fixed point of $T_c$ of degree zero, then $\omega \in K u$.
\end{lm}
\proof{First we show that for any element  $z=h_1\cdots h_k g_k^*\cdots g_1^*$ the following dichotomy holds: either there is an $m$ such that $T_c^m(z)=0$ or there is an $m$ such that $T_c^m(z)=u$. To prove this, consider  $z=h_1\cdots h_k g_k^*\cdots g_1^*$ such that $T_c^m(z)\ne 0$ for each $m$. Then $0 \ne T_c(z)=e_n^*\dots e_1^* h_1\cdots h_k g_k^*\cdots g_1^* e_1 \dots e_n$.  
\newline
If $n> k$ then $h_i=g_i=e_i$ for all $i \in \{1,\dots k\}$ and $T_c(z)=e_n^*\dots e_{k+1}^*e_{k+1}\dots e_{n}=u$. If $n\le k$, $k=qn+r$, then $T_c^{nq+1}(z)=(c^*)^{nq+1}h_1\cdots h_k g_k^*\cdots g_1^*c^{nq+1}=e_n^*\dots e_{r+1}^*e_{r+1}\dots e_n=u$. 
\newline \indent
Now, consider a degree zero element $\omega$ such that $T_c(\omega)=\omega$. We can write $\omega=\sum_{i \in I} t_i \mu_i \tau_i^*$, with $t_i \in K$ and $\mu_i, \tau_i \in \Path(E)$ such that $l(\mu_i)=l(\tau_i)$. Then $I=A\buildrel\cdot\over\cup B$, where $A$ is the set of all $i \in I$ such that $\mu_i, \tau_i^* \in \ker(T_c^m)$ for some $m$ and $B=I\setminus A$. Since $I$ is finite there exists an $m_1$ such that $T_c^{m_1}(\mu_i\tau_i^*)=0$ for each $i \in A$. For the same reason and taking into account the proved dichotomy there is an $m_2$ such that $T_c^{m_2}(\mu_i\tau_i^*)=u$ for each $i \in B$. Thus, defining $m:=\max (m_1,m_2)$ we have 
$\omega=T_c^m(\omega)=(\sum_{i \in B} t_i)u \in Ku$.}

\begin{co} Let $E$ be a graph and $A=L_K(E)$ or $C_K(E)$.  Take $z\in Z_0=Z(A)_0$ with Peirce decomposition $z=\sum_{v \in E^0} z_v$. If there exists a closed path $c$ based at a vertex $u$, then $z_u\in
Ku$.\end{co}
\proof{Since $zc=cz$ and $s(c)=r(c)=u$ we have $z_uc=cz_u$. Moreover,
$T_c(z_u)=c^*z_uc=c^*cz_u=z_u$ and by the previous lemma we conclude $z_u\in Ku$.}
\medskip

Now, suppose we are in the following situation:

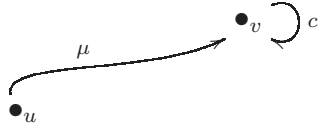
\[
\xygraph{
!{<0cm,0cm>;<1.5cm,0cm>:<0cm,1.2cm>::}
!{(0,0)}*+{\bullet_{u}}="u"
!{(2,1)}*+{\bullet_{v}}="v"
"u":@(lu,ld)"v"^{\mu}
"v":@(ru,rd)"v"^{c}
}
\]
where $u ,\ v \in E^0$, $c$ is a cycle based at $v$, $\mu=f_1\cdots f_k$ is a path and there is no closed path based at $s(f_i)$ for all $i \in \{1,\cdots ,k\}$. Take $z \in Z(A)_0, \ z=\sum_{w \in E^0}z_w$; then $\mu^*z_u\mu = \mu^*\mu z_v=z_v$ and since $z_vc=cz_v$ we have $z_v \in \text{Fix}(T_c)$, hence $z_v=l v$ for some $l \in K^{\times}$; therefore
\begin{equation}\label{parasi}
\mu^*z_u\mu =lv.
\end{equation}

\begin{lm}
With the notation above, consider $z_u=ku+\sum_il_i\sigma_i\sigma_i^*+\sum_j m_j \mu \gamma_j \gamma_j^*\mu^*$ with $k,\ l_i, m_j \in K$, $u \in E^0$, $\sigma_i, \mu, \gamma_i \in \Path(E)$,  $\{u\}\cup \{\sigma_i\sigma_i^*\}\cup \{\mu \gamma_j \gamma_j^*\mu^*\}$ a linear independent set and the paths $\sigma_i$  not of the form $\mu \gamma_j$. Then  $\sum_j m_j \mu \gamma_j \gamma_j^*\mu^*=h_0\mu\mu^*$ for some scalar $h_0 \in K^{\times}$.
\end{lm}
\proof{From the expression  of $z_u$ we get $z_u\mu=k\mu+\sum_il_i\sigma_i\sigma_i^*\mu+\sum_j m_j \mu \gamma_j \gamma_j^*$. It is easy to prove that if $\sigma_i\sigma_i^*\mu \ne 0$ then $\sigma_i\sigma_i^*\mu=\mu$, so $z_u\mu=k\mu+\sum_il_i\mu+\sum_jm_j\mu\gamma_j\gamma_j^*$ (observe that in this last sum with $l_i$'s we may have less summands than in the original one). Since $\mu^*z_u\mu=kv+\sum_il_iv+\sum_jm_j\gamma_j\gamma_j^*$ and by the formula (\ref{parasi}) we have  $\mu^*z_u\mu=kv+\sum_il_iv+\sum_jm_j\gamma_j\gamma_j^*=\lambda v$, which proves the lemma.
}
\begin{pr}\label{finited}
 Let $E$ be a finite and row-finite graph and $A=L_K(E)$ or $C_K(E)$. Then  $\text{dim}_K(Z_0)$ is finite.
\end{pr}
\proof{We deduce from the  lemma above that if there exists a path $\mu$ such that  $s(\mu)=u$ and $\mu$ connects to a cycle $c$ then, in the expression  of $z_u=ku+\sum_il_i\sigma_i\sigma_i^*+\lambda_0\mu\mu^*$, the
cycle $c$ does not appear. We can argue as before if we consider any other cycle and path with source $u$ connecting to the cycle.  Since we deal with finite graphs  the set of paths which may appear in the expression of $z_u$ is finite and so
 $\dim_K(Z_0)$ is finite.}

\begin{lm}\label{capu}
If $A=L_K(E)$ or $C_K(E)$ is prime then $Z_0=Z(A)_0=K$.
\end{lm}
\proof{If $K$ is algebraically closed then $Z(A)_0$ is a finite dimensional algebra by Proposition \ref{finited} and as $A$ is prime then  $Z(A)_0$ is a domain so $Z(A)_0=K$. If $K$ is not algebraically closed, let $\Omega$ be the algebraic closure of $K$, then $Z(A_\Omega )=Z(A)\otimes \Omega$ and $Z(A_\Omega)_0=Z(A)_0 \otimes \Omega$. Besides, if $A=L_K(E)$, then  $A_\Omega = A\otimes_K  \Omega=L_K(E)\otimes \Omega \cong L_\Omega(E)$ hence $A_\Omega $ is prime (because the primeness condition is given by a property of the graph). If $A=C_K(E)$ then   $A_\Omega = A\otimes_K  \Omega=C_K(E)\otimes \Omega \cong C_\Omega(E)$ hence $A_\Omega$ is also  prime  as before.
In any case $1=\dim_\Omega(Z(A_\Omega))_0=\dim_K(Z(A)_0)$ so that  $Z(A)_0=K$.}
\subsection{The center of  a prime Cohn path algebra}

Next we will study the center of a prime Cohn path algebra $C_K(E)$. 

 \begin{pr}
 If $C_K(E)$ is prime then $E$ is the $m$-petals rose graph and $Z(C_K(E))=K$.
 \end{pr}
   \proof{The primeness of $C_K(E)$  and the finiteness of $E$ imply that $Z(C_K(E))$ is  a domain. By Lemma \ref{capu} we know that $Z(C_K(E))_0=K$ is a field. So we can apply Theorem \ref{padespues} to $C_K(E)$ with its standard $\Z$-grading and involution.  We conclude that 
$Z\cong K$ or $Z \cong K[x,x^{-1}]$. Now we must discard the second possibility. By Theorem \ref{cohnprima} the graph $E$ must be the $m$-petals rose  $R_m$. 
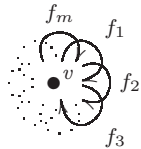
 $$\xymatrixcolsep{3.3pc} \xymatrix@-1.0pc{\bullet^{\emph{v}}\ar@(l,u)@{.}[]\ar@(ul,ur)[]^{f_{m}} \ar@(u,r)[]^{f_{1}}\ar@(ur,dr)[]^{f_{2}}\ar@(r,d)[]^{f_{3}}\ar@(dr,dl)@{.}\ar@(l,d)@{.}\ar@(dl,ul)@{.}} $$ 
If $m=0$ then $C_K(R_0)=Kv$ and there is nothing to prove. So we assume $m\ge 1$.
Let $z \in Z_n$ with $n>0$, then $z=l_0 \mu_0+ l_1 \mu_1(\gamma_1)^*+ l_2 \mu_2(\gamma_2)^*+\cdots+l_r \mu_r(\gamma_r)^*$, with  $l_i \in K$, $l_r \in K^{\times}$, $\mu_i, \gamma_i  \in \Path(E)$, $\partial_{\mathbb{R}}(\mu_0)=n$, $\partial_{\mathbb{R}}(\mu_1)=n+1, \cdots, \partial_{\mathbb{R}}(\mu_j)=n+j$, and $\partial_{\mathbb{R}}(\gamma_1)=1, \cdots, \partial_{\mathbb{R}}(\gamma_j)=j$.
We have $$ \mu^*_0z=l_0 + l_1 \mu_0^*\mu_1(\gamma_1)^*+ l_2 \mu_0^*\mu_2(\gamma_2)^*+\cdots+l_r \mu_0^*\mu_r(\gamma_r)^*,$$
where $\partial_{\mathbb{R}}(\mu^*_0z)\leq r.$ On the other hand
$$z\mu^*_0=l_0 \mu_0\mu^*_0+ l_1 \mu_1(\gamma_1)^*\mu^*_0+ l_2 \mu_2(\gamma_2)^*\mu^*_0+\cdots+l_r \mu_r(\gamma_r)^*\mu^*_0$$
  and now  $\max(\partial_{\mathbb{R}}(z\mu^*_0))=n+ r$, whence $n+r \leq r$, which implies $n \leq 0$, a contradiction, therefore $z=0$.  For $ z \in Z_{-n}$, with $n>0$, we may apply the involution to get again $z=0$. Thus we conclude $Z=Z_0=K$.}

\subsection{The center of a Prime Leavitt path algebra}

We begin this subsection by introducing some definitions. 
For a path $\mu=e_1\dots e_n$ we denote by $\mu^0$ the set of vertices given by $\mu^0:=\{s(e_i) \ \colon \ i= 1, \dots n\}\cup \{r(e_n)\}$.
\medskip

For $X$ a nonempty subset of an algebra $A$, we denote by $I(X)$ the ideal generated in $A$ by $X$.
\medskip

For a graph $E$ we will denote by $P_c(E)$ the set of vertices given by $P_c(E):=\cup \mu^0$, where $\mu$ ranges in the
set of all cycles without exits of the graph. Recall also the so called Condition (L): we say that a graph $E$ satisfies \emph{Condition} (L) if each cycle in $E$ has an exit.

\medskip

The proof of the following result is contained in \cite[Proposition 3.5]{AAPS}.

\begin{pr}\label{IdCicSins} 
Let $c$ and $d$ be cycles without exits in a graph $E$. Then:
\begin{enumerate}
\item[\rm (i)] $I(c^0) \cong M_{n}(K[x, x^{-1}])$ (matrices with only a finite number of nonzero entries), where $n$ is the cardinal of the set of paths ending at the cycle $c$ and not containing all the edges of this cycle.
\item[\rm (ii)] $I(c^0) I(d^0)=0$.
\end{enumerate}
\end{pr}

\begin{theo}[\textbf{Reduction Theorem}]\label{3.1} {\rm (\cite[Proposition 3.1]{AMMS1})}.
Let $E$ be an arbitrary graph. Then for every nonzero element $z\in L_K(E)$ there exist $\mu,\nu \in$ Path$(E)$ such that:

\begin{enumerate}
\item[{\rm (i)}] $\mu^*z\nu = kv$ for some $k\in K\setminus\{0\}$ and $v\in E^0$, or
\item[{\rm (ii)}] there exists a vertex $w\in P_c(E)$ such that $\mu^*z\nu$  is a nonzero polynomial $p(c, c^*)$, where $p(x, x^{-1}) \in K[x,x^{-1}]$.
\end{enumerate}

\noindent Both cases are not mutually exclusive.
\end{theo}

\begin{theo}\label{atomo}
Let $E$ be a  graph such that $L_K(E)$ is a prime Leavitt path algebra, then $Z(L_K(E))\ne 0$ if and only if $\vert E^0\vert < \infty$. In this case:
\begin{enumerate}
\item[{\rm (i)}] $Z(L_K(E))\cong K$ if and only if $E$ satisfies Condition {\rm (L)} or there exists a unique cycle without exits $c$ and infinitely many paths ending at $c$ and not containing $c$.
\item[{\rm (ii)}] $Z(L_K(E))\cong K[x, x^{-1}]$ if and only if $E$ contains a unique cycle without exits and there are only a finite number of paths ending at $c$ and not containing $c$.
\end{enumerate}
\end{theo}
\begin{proof}{
We observe first that if there are cycles without exits, then there is only one, since the Leavitt path algebra is prime and by condition (ii) in Proposition  \ref{IdCicSins}.
\newline
By Proposition \ref{finito}, $Z(L_K(E))\ne 0$ if and only if $\vert E^0\vert < \infty$. Now suppose $\vert E^0\vert < \infty$.
Applying Theorem \ref{padespues} and Lemma \ref{capu}, we know that  the center of $L_K(E)$ is isomorphic to $K$ or to $K[x, x^{-1}]$. 
Suppose first that $Z(L_K(E))\cong K$. We will see that $E$ satisfies Condition (L) or there exist a unique cycle without exits and infinite paths
ending at $c$ and not containing $c$.
\newline
Supose that $c$ is a cycle without exits, and let $I$ be the (graded) ideal generated by $c^0$. By condition (i) in Proposition \ref{IdCicSins} the ideal  $I$ is isomorphic to $M_{n}(K[x, x^{-1}])$.
 Consider first $n\in \mathbb{N}\setminus\{0\}$. By \cite[Remark 3.4]{BPSS} we have $Z(I)=Z(L_K(E))\cap I$. Since
 $Z(I)$ isomorphic to $K[x, x^{-1}]$, which is infinite dimensional as a $K$-vector space and $Z(L_K(E))\cong K$, we have $Z(L_K(E)) \cap I$ is 1-dimensional over $K$, a contradiction. 
\newline
Now, suppose $Z(L_K(E))\cong K[x, x^{-1}]$, and let $\varphi$ be an isomorphism from $Z(L_K(E))$ into $K[x, x^{-1}]$. Take $a\in Z(L_K(E))$ such that $\varphi(a)=x$;
by the proof of Theorem \ref{padespues}, we may suppose that $a$ is a homogeneous element of degree $d>0$. Then:
\newline
$$(\dag)\quad  a a^* =a^* a = 1.$$
\newline
We claim that there exists a vertex $v$ in $E^0$ and $k\in K^\times$ such that $kva$ or $kva^*$ is a power of a cycle without exits. Indeed, by the Reduction Theorem there exist paths $\mu, \nu$ such that $\mu^*a\nu = kv$ for some $k\in K\setminus\{0\}$ and $v\in E^0$, or
$\mu^*a\nu$ is a nonzero polynomial $p(c, c^*)$, for $p(x, x^{-1})\in  K[x,x^{-1}]$.
\newline
In the first case, $\mu^*a\nu = kv$; apply that $a$ is in the center of $L_K(E)$ to get: $a\mu^*\nu =kv$. Multiply by $a^*$ on the left hand side of each term of this identity and apply $(\dag)$ to get: $\mu^*\nu =kva^*$ (note that $a^* \in Z(L_K(E))$). Since the degree of $a$ is $d >0$, then $(\mu^*\nu)^*= \nu^*\mu=kva$ is a path. 
\newline
Because $a\in Z(L_K(E))$, $kva=\nu^*\mu$ is a closed path starting and ending at $v$, say $\nu^*\mu=e_1\dots e_n$. We claim this closed path has no exits. Suppose on the contrary that there exists $i\in \{1, \dots, n\}$ and $f\in E^1$ such that $s(f)=s(e_i)$ and $f\neq e_i$. Denote by $w$ the source of $e_{i}$. Then:
$$e_i \dots e_n e_1\dots e_{i-1} =
(e_{i-1}^* \dots e_1^*) e_1\dots e_{i-1}e_i \dots e_n (e_1\dots e_{i-1})= $$ $$(e_{i-1}^* \dots e_1^*)kva (e_1\dots e_{i-1}) 
= kwa.$$
This implies
$0 = f^*e_i \dots e_n e_1\dots e_{i-1} = f^*(kwa)= kf^*a$. Multiply by $a^*$ and use again $(\dag)$. We get $0=kf^*$, a contradiction, and we have proved that $kva$ is a cycle without exits. 
\newline
If we are in the second case, that is, $\mu^*a\nu$ is a nonzero polynomial $p(c, c^*)$, we arrive again to the existence of a cycle without exits in $E$.
\newline
Finally we show that the number of paths ending at $c$ and not containing $c$ is finite.
Suppose on the contrary that there are infinitely many paths in this situation. Since the number of vertices is finite, this means that there exists a cycle with exits $d$ connecting to $c$. Let $u=s(d)$. A generator system for $uL_K(E)u$ is $A\cup B$, for 
$$A= \left\{
d^n(d^m)^*\ \vert \ n, m \geq 0\right\}$$
and
$$
B= \left\{
d^n\alpha \beta^*(d^m)^*\ \vert \ n, m \geq 0,\ s(\alpha)=u=s(\beta), d \not\leq \alpha,\ d \not\leq \beta,\ \alpha^1\cup\beta^1\not\subset d^1 \right\}.$$
For $n=0$ we understand $d^n=u$. Note that given $n,m \geq 0$ and $d^n\alpha \beta^*(d^m)^*\in B$, there exists a  suitable $r\in \mathbb{N}$ such that $(d^r)^*d^n\alpha \beta^*(d^m)^*d^r=0$. This gives us that if we define the map $S:L_K(E) \to L_K(E)$  by
$S(x) = d^\ast x d$,  for every $b\in B$ there is an $n\in \mathbb{N}$ satisfying $S^n(b)=0$. Note that $au$ is a fixed point for $S$. A consequence of this reasoning is that $au\in span(A)$. Write $au=\sum_nk_{n}d^n(d^m)^*$, for $m=n-deg(a)$. Then, for some $l\in \mathbb{N}$ we have $au=S^l(au)=\sum_nk_nd^{deg(a)}$. Since $au$ commutes with every element in $uL_K(E)u$, the same should happen to $d^{deg(a)}$, but this is not true as it does not commute with $d^*$, giving $au=0$ and therefore $u=a^*au=0$ (by $(\dag)$), a contradiction.
}
\end{proof}

\section{Consequences for the center of a general Leavitt path algebra}

Once we know how to compute the center of a prime Leavitt path algebra, we would like to get an idea, as close as possible, on the structure of the center of a general Leavitt path algebra associated to a row-finite graph $E$. In order to achieve this target, we will give a structure theorem for Leavitt path algebras in terms of subdirect products of prime Leavitt path algebras (note that every Leavitt path algebra is a semiprime associative algebra; see, for example, \cite[Proposition 1.1]{AMMS1}). The key notion will be that of graded Baer radical whose behaviour is quite similar to that of the Baer radical (in a nongraded sense).

\medskip
In this section we shall work with $\Z$-graded algebras and since no other grading group will be considered, the term \lq\lq graded algebra\rq\rq\ will mean $\Z$-graded algebra.
\medskip

Let $A$ be a graded algebra, $A=\oplus_{n\in \Z}A_n$. Recall that $A$ is said to be \emph{graded semiprime} when for any graded ideal $I$ of $A$ we have $$I^2=0\text{ implies } I=0.$$ Also $A$ is said to be \emph{graded prime} when for any two graded ideals $I$ and $J$ of $A$ we have 
$$IJ=0\text{ implies } I=0 \text{ or } J=0.$$
It is easy to prove that $A$ is graded semiprime if and only if for any
homogeneous element $x\in A$ we have that $xAx=0$ implies $x=0$ (which is the definition of being graded non degenerate).
It is also straightforward to see that $A$ is graded prime if and only if
for any two homogeneous elements $x,y\in A$ we have that $xAy=0$ implies
$x=0$ or $y=0$. As a corollary we have:
$$A \text{ is prime } \text{ if and only if } A \text{ is graded prime,} $$
$$A \text{ is semiprime } \text{ if and only if } A \text{ is graded semiprime.} $$
\medskip

The proofs of these facts can be seen in \cite[Proposition II.1.4 (1)]{NvO}.
\medskip

Recall that a \emph{graded prime ideal} $I$ of $A$ is graded ideal
such that $A/I$ is a  graded prime algebra. In the same way  we can define
the notion of graded semiprime ideal.
\begin{de}\rm
Let $A$ be a graded algebra; we define the \emph{graded Baer radical}
of $A$ as the intersection of all graded prime  ideals of $A$. We shall denote
it by $\rad(A)$.
\end{de}
The following theorem is the graded version of a well-known result on the classical Baer radical. Its proof runs parallel to that of the classical theorem; we include it here for completeness.

\begin{theo}
The graded Baer  radical of a graded algebra $A$ is a graded semiprime ideal. In fact, it is the least graded ideal which is semiprime.
\end{theo}
\proof{ If $x \in A$ is such that $xAx\subset\rad(A)$ then $xAx\subset P$ for any graded prime ideal
$P$. Since $P$ is semiprime we have $x\in P$ hence $x\in\rad(A)$. To prove that $\rad(A)$ is contained in any graded semiprime ideal $I$ we consider the class
$\{P_\a\}_{\a}$ of all graded prime ideals of $A$  which contain $I$. Next we prove that $I=\cap_\a P_\a$. If there is some $x\in\cap_\a P_\a$ such that
$x\not\in I$ we define $x_0:=x$. Then $x_0Ax_0\not\subset I$ (because $I$
is semiprime) and $x_0Ax_0\subset\cap_\a P_\a$. Thus we can take some
$x_1\in x_0 Ax_0$ such that $x_1\notin I$, $x_1\in P_\a$ (for any $\a$). We  repeat this argument to obtain a sequence $X:=\{x_0,x_1,\ldots, x_n,x_{n+1},\ldots\}$ such that $x_n\notin I$, $x_n\in\cap_\a P_\a$ and $x_{n+1}\in x_nAx_n$. Define the family $\mathfrak{F}$ of all graded ideals $J$ of $A$
such that $I\subset J$ and $J\cap X=\emptyset$. Since $I\in\mathfrak{F}$ we have $\mathfrak{F}\ne\emptyset$. Thus by Zorn's Lemma we have a maximal element $P\in\mathfrak{F}$. Let us prove that $P$ is a prime ideal. In order to do that take two graded ideals $J_i$ of $A$ such that $P\subset J_i$ for $i=1,2$ and
$J_1J_2\subset P$. We have to prove that $P=J_i$ for some $i=1,2$. 
Suppose on the contrary $P\subsetneq J_i$; this implies $J_i\cap X\ne\emptyset$ and so
some $x_n\in J_1$ and some $x_m\in J_2$. Thus, if $l\ge\max(m,n)$ we
have $x_l\in J_1\cap J_2$ and $x_{l+1}\in J_1J_2\subset P$. But $P\in\mathfrak{F}$, which is a contradiction. So we have proved that $P$ is a prime ideal, hence it is a graded prime ideal containing $I$. Since each $x_n\in\cap_\a P_\a\subset P$ we have $x_n\in P$, which contradicts the fact that
$P\cap X=\emptyset$. This proves the theorem.}
\medskip
\begin{de}
Let $A$ be a graded algebra then we denote by $\P$ the family of all graded prime ideals of $A$. 
\end{de}

\begin{co}\label{ocho}
Let $L_K(E)$ be a Leavitt path algebra. Then:
\begin{enumerate}
\item[{\rm (i)}] $\rad(L_K(E))=0.$
\item[{\rm (ii)}] $L_K(E)$ is a subdirect product of prime Leavitt path algebras.
\end{enumerate}
\end{co}\label{baercero}
\proof{ Given that $A:=L_K(E)$ is graded semiprime, $0$ is a graded semiprime ideal and of course it is the least graded ideal which is semiprime, so $\rad(A)=0$. To prove the second assertion, take into account that $\cap_{P \in \Ps}P=\rad(A)=0$. Then the map 
\newline
$$\begin{matrix}j\colon A\to\prod_{P\in\Ps} A/P\cr
a\mapsto (a+P)_P
\end{matrix}$$ 
\newline
is a monomorphism. 
\newline
For any $Q\in\P$ let $\pi_Q\colon\prod_P A/P
\to A/Q$ be the canonical projection. Since the composition $\pi_Q j$
is an epimorphism then $A$ is the subdirect product of the $A/P$, which are
graded prime algebras. Moreover, since each $P$ is a graded ideal then it is the ideal generated by some  hereditary and saturated subset $H_P$ in $E^0$ (see \cite[Lemma 2.1 and Remark 2.2]{APS1}).  Therefore $A/P\cong L_K(E/H_P)$, where $E/H_P$ denotes the quotient graph (see \cite[Lemma 2.3 (1)]{APS1}). So $A$ is the subdirect product of the Leavitt path algebras $L_K(E/H_P )$.}
\medskip
 
Take now $A$ to be the Leavitt path algebra $L_K(E)$. As it is known (see \cite[Remark 2.2]{APS1}), the graded  ideals and, in particular, the graded prime ideals $P\in\P$ are of the 
form $P=I(H)$ for a unique $H\in\mathcal H_E$, where $\mathcal H_E$ stands for the hereditary and saturated subsets of $E^0$. The quotient algebra $A/P\cong L_K(E/H)$ is prime, hence
the quotient graph $E/H$ is downward directed. Therefore there are three``mutually excluding possibilities" for the graph $E/H$:
either it satisfies condition (L) or it has a unique cycle without exits and a finite number of paths ending at this unique cycle and not containing it or there is a unique cycle without exits and an infinite number of paths ending at this cycle and not containing it.  Thus we can classify the prime ideals $P\in\P$
into two flavours:
\smallskip
$$\begin{aligned}
{\mathcal  I} = & \{ I(H) \colon 
H\in\mathcal H_E, \
E/H \text{ is downward directed and satisfies Condition (L)}\}
\cup\\
& \{I(H) \colon 
H\in\mathcal H_E, \
E/H \text{ is downward directed, has a unique cycle without exits} 
\\
&\text{and infinitely many paths ending at this cycle and not containing it}\}
\end{aligned}
$$
\smallskip
$$\begin{aligned}
\small{\mathcal J} = &\{ I(H) \colon H\in\mathcal H_E,\ E/H\ \text{is downward directed, has a unique cycle without exits}\\
& \text{and there is a finite number of paths ending at this cycle and not containing it}\}
\end{aligned}
$$

\begin{theo}\label{molecula}
 For a  row-finite graph $E$,
the center of $L_K(E)$  is a subalgebra of $\prod_{P \in {\mathcal I}} K_P \times \prod_{Q \in {\mathcal J}} K_{Q}[x,x^{-1}]$   containing the ideal $\bigoplus_{P\in \Ps} Z(W_P)$,  i.e.: 
$$\bigoplus_{P\in \Ps} Z(W_P)\triangleleft Z(L_K(E))\subset \prod_{P \in {\mathcal I}} K_P \times \prod_{Q \in {\mathcal J}} K_{Q}[x,x^{-1}],$$ 
where:
\begin{enumerate}
\item[{\rm (i)}] $K_P =K_{Q}=K$ for any $P\in\mathcal I$ and $Q\in\mathcal J$.
\item[{\rm (ii)}]  For any $P\in\P$ the ideal $W_P$ is defined as the intersection of all the graded prime ideals others than $P$.
\end{enumerate} 
\end{theo}
\proof{Let $A$ be as before the Leavitt path algebra $L_K(E)$. To prove that $Z(A)$ is a subalgebra of $\prod_{P \in {\mathcal I}} K_P \times \prod_{Q \in {\mathcal J}} K_{Q}[x,x^{-1}]$,
consider  $z\in Z(A)$ and let $j$ and $\pi_Q$ be as in the proof of Corollary \ref{ocho}. We want to show that $j(a)$ is in the center of  $\prod_P A/P$. 
Since $\pi_Q j$ is an epimorphism for any $Q$, we have
$\pi_Q j(a)\in Z(A/Q)$, so $a+Q\in Z(A/Q)$. Consequently,
$j(a)\in Z(\prod_P A/P)$, which, up to isomorphism, is of the form $\prod_{P \in {\mathcal I}} K_P \times \prod_{Q \in {\mathcal J}} K_Q[x,x^{-1}]$  by Theorem \ref{atomo}.
\newline
In order to prove the second assertion, we prove that the sum of the ideals $W_P$ is direct. Since $W_P=\cap_{Q\in\Ps\setminus\{P\}} Q$ we have $W_P\subset Q$ for any $Q\in\P$, $Q\ne P$. So, for any $P\in\P$ we have
$W_P\cap(\sum_{Q\ne P} W_Q)\subset (\cap_{R\in\Ps\setminus\{P\}} R)\cap P=\rad(L_K(E))=0$. 
\newline
To finish the proof take into account that $Z(\oplus_P W_P)=\oplus_P Z(W_P)$ and that the center of an ideal of a semiprime algebra is contained in the center of the algebra.
}
\bigskip

The upper bound for $Z(L_K(E))$ given in Theorem \ref{molecula} and Theorem \ref{atomo} allows to say that
 the building blocks for this upper bound are $K$ and $K[x,x^{-1}]$. 
The number of $K$'s and $K[x,x^{-1}]$'s appearing is completely determined by the cardinal of the sets $\mathcal I$ and $\mathcal J$. Thus it is easily computable directly from the graph $E$.
On the other hand the lower bound is also algorithmically computable for a given finite graph since each $W_P$ is an
intersection of ideals generated by hereditary and saturated subsets of the graph (hence each $W_P$ is also the
ideal generated by some hereditary and saturated set which can be determined from the graph). So, again the
ideals $W_P$ are Leavitt path algebras.  
\smallskip

We have checked with several examples of concrete graphs that the center of a Leavitt path algebra may agree with the lower or with the upper bound described in Theorem \ref{molecula}. However our approach, the precise structure of the center remains an open question.   

\section*{Acknowledgments}
All the authors have been partially supported by the Spanish MEC and Fondos FEDER through project MTM2010-15223,  by the Junta de Andaluc\'{\i}a and Fondos FEDER, jointly, through projects FQM-336, FQM-02467 and FQM-3737 and by the  programa de becas para estudios doctorales y postdoctorales SENACYT-IFARHU, contrato no. 270-2008-407,  Gobierno de Panam\'a. This work was done
during research stays of the first and last author in the University of M\'alaga. Both
authors would like to thank the host center for its hospitality and support.

\end{document}